\documentclass[a4paper, 11pt]{amsart}
\usepackage[T1]{fontenc}
\usepackage{amssymb,amsthm,amsmath}
\usepackage{enumitem}
\usepackage{amsaddr}
\usepackage{eucal}

\newtheorem{thm}{Theorem}

\newtheorem{lem}[thm]{Lemma}

\newtheorem{fact}[thm]{Fact}

\theoremstyle{definition}
\newtheorem{defi}[thm]{Definition}
\newtheorem{exam}[thm]{Example}
\newtheorem{rem}[thm]{Remark}

\numberwithin{equation}{section}
\newcommand{\mb}{\mathbb}

\DeclareMathOperator{\ext}{ext}
\DeclareMathOperator{\id}{id}

\def\mc{\mathcal}
\def\mb{\mathbb}

\begin{document}
\title[Hofmann-Lawson duality for locally small spaces]{Hofmann-Lawson  duality for locally small spaces}
\author[A. Pi\k{e}kosz]{Artur Pi\k{e}kosz} 
\address{{\rm Department of Applied Mathematics, Cracow University of Technology,\\
 ul. Warszawska 24, 31-155 Krak\'ow,  Poland\\
 email: \textit{pupiekos@cyfronet.pl}
 }}

\date{\today}
\keywords{Hofmann-Lawson duality, Stone-type duality, locally small space, spatial frame, continuous frame.}

\subjclass[2010]{Primary: 06D22, 06D50. Secondary: 54A05, 18F10.}

\begin{abstract}
We prove versions of the spectral adjunction, a Stone-type duality and Hofmann-Lawson duality
for locally small spaces with bounded continuous mappings. 
\end{abstract}
\maketitle
\section{Introduction}
Hofmann-Lawson duality belongs to the most important dualities in lattice theory.
It was stated in \cite{HL} (with the reference to the proof in \cite{HK}).
Its versions for various types of structures are numerous (see, for example, \cite{E}).

The concept of a  locally small space comes from that of  Grothendieck topology through  generalized topological spaces in the sense of Delfs and Knebusch (see \cite{DK,P1,P}). Locally small spaces were used in o-minimal homotopy theory (\cite{DK,P3}) as underlying structures of locally definable spaces.  A simple language for locally small spaces was introduced and used in \cite{P2} and \cite{P}, compare also \cite{PW}. It is analogical to  the language of  Lugojan's generalized topology (\cite{Lu}) or Cs\'{a}sz\'{a}r's generalized topology (\cite{C}) where a family of subsets of the underlying set is satisfying some, but not all, conditions for a topology.
However treating locally small spaces as topological spaces with additional structure seems to be  more useful. 

While Stone duality for locally small spaces  was discussed in \cite{P-stone}, we consider in this paper the spectral adjunction, a Stone-type duality and Hofmann-Lawson duality. 

The set-theoretic axiomatics of this paper is what Saunders Mac Lane calls the standard
Zermelo-Fraenkel axioms (actually with the axiom of choice) for set theory plus the existence of a set which is a universe, see \cite{MacLane}, page 23.

\textbf{Notation.}
We shall use a special notation for family intersection
$$ \mc{U} \cap_1 \mc{V} = \{ U \cap  V : U \in \mc{U}, V \in \mc{V} \}.$$ 


\section{Preliminaries on frames and their spectra.}
First,  we set the notation and collect basic material on  frame theory  that can be found in \cite{G2, CLD, HL, PP,PS}.

\begin{defi}
A \textit{frame} is a complete distributive lattice $L$ satisfying the (possibly infinite) distributive law $a\wedge \bigvee_{i\in I} b_i =\bigvee_{i\in I} (a\wedge b_i)$
for $a, b_i \in L$ and $i\in I$.
\end{defi}
\begin{defi}A \textit{frame homomorphism}  is a lattice homomorphism between frames preserving all joins. The category of frames and frame homomorphisms will be denoted by 
$\mathbf{Frm}$. The category of frames and right Galois adjoints of frame homomorphisms will be denoted by $\mathbf{Loc}$.
\end{defi}

\begin{defi}
For a frame $L$, its \textit{spectrum} $Spec(L)$ is the set of non-unit primes of $L$:
$$Spec(L)=\{ p\in L\setminus \{1\}:(p=a\wedge b)\implies (p=a \mbox{ or } p=b)\}.$$
For $a\in L$, we define
$$\Delta_L(a)=\{ p\in Spec(L): a\not\leq p\}.$$
For a subset $S\subseteq L$, we set
$$\Delta_L(S)=\{\Delta_L(a):a\in S \}.$$
\end{defi}

\begin{fact}[{{\cite[Prop. V-4.2]{G2}}}] \label{topol}
For any frame $L$ and $a,b,a_i \in L$ for $i\in I$, we have
$$ \Delta_L(0)=\emptyset, \qquad \Delta_L(1)=Spec(L),$$
$$\Delta_L(\bigvee_{i\in I}  a_i) = \bigcup_{i\in I} \Delta_L(a_i), \quad \Delta_L(a\wedge b)=\Delta_L(a) \cap \Delta_L(b).$$
Consequently, the mapping $\Delta_L: L\ni a\mapsto \Delta_L(a)\in \Delta_L(L)$ is a surjective frame homomorphism.
\end{fact}

\begin{defi}  
The set $\Delta_L (L)$ is a topology on $Spec(L)$, called the \textit{hull-kernel topology}.
\end{defi}

\begin{rem}
In this paper (as opposed to \cite{P-stone}), being sober implies being $T_0$ (Kolmogorov).
\end{rem}

\begin{fact}[{{\cite[Ch. II, Prop. 6.1]{PP}, \cite[V-4.4]{G2}}}] \label{sober}
For any frame $L$,  the topological space $(Spec(L),\Delta_L(L))$ is   sober.
\end{fact}

\begin{defi}
A frame $L$ is \textit{spatial} if
 $Spec(L)$ \textit{order generates} (or: \textit{inf-generates}) $L$, which means that
 each element of $L$ is a meet of primes.
\end{defi}

\begin{fact}[{{\cite[Ch. II, Prop. 5.1]{PP}}}] \label{iso-spatial}
If $L$ is a spatial frame, then $\Delta_L$ is an isomorphism of frames.
\end{fact}

\begin{defi}[{{\cite[p. 286]{HL}}}]
For $a,b \in L$, we say that  $b$ is \textit{well-below} $a$ (or: $b$ is \textit{way below} $a$) and write $b \ll a$ if for each (up-)directed set $D\subseteq L$ such that $a\leq \sup D$ there exists $d\in D$ such that $b\leq d$. 
\end{defi}
\begin{defi}
A frame is called \textit{continuous} if for each element $a\in L$ we have
$$a=\bigvee \{ b\in  L: b \ll a   \}.$$
\end{defi}

\begin{fact}[{{\cite[Ch. VII, Prop. 6.3.3]{PP}, \cite[I-3.10]{G2}}}] 
Every continuous frame is spatial.
\end{fact}

The following two facts lead to Hofmann-Lawson duality.

\begin{fact}[{{\cite[Thm. V-5.5]{G2}, \cite[Ch. VII, 6.4.2]{PP}}}] \label{6.4.2}
If $L$ is  a  continuous frame, then $(Spec(L),\Delta_L(L))$ is a sober locally compact topological space
and $\Delta_L$ is an isomorphism of frames.
\end{fact}

\begin{fact}[{{\cite[Thm. V-5.6]{G2}, \cite[Ch. VII, 5.1.1]{PP}}}] \label{5.1.1}
If $(X,\tau_X)$ is a sober locally compact topological space, then $\tau_X$ is a continuous frame.
\end{fact}



Finally, we recall the classical Hofmann-Lawson duality theorem.

\begin{thm}[{{\cite[Thm. 2-7.9]{PS}}}, {{\cite[Thm. 9.6]{HL}}}] \label{HL}
The category $\mathbf{ContFrm}$ of  continuous frames with frame homomorphisms and the category $\mathbf{LKSob}$ of locally compact sober spaces and continuous maps are dually equivalent. 
\end{thm}

\section{Categories of consideration.}

Now we present the basic facts on the theory of locally small spaces.
\begin{defi}[{{\cite[Definition 2.1]{P}}}]\label{lss}
A \textit{locally small space} 
  is a pair $(X,\mc{L}_X)$, where $X$ is any set and  $\mc{L}_X\subseteq  \mc{P}(X)$ satisfies the following conditions:
\begin{enumerate}
\item[(LS1)] \quad $\emptyset \in \mc{L}_X$,
\item[(LS2)] \quad if $L,M \in \mc{L}_X$, then $L\cap M, L\cup M \in \mc{L}_X$,
\item[(LS3)] \quad $\forall x \in X \: \exists L_x \in \mc{L}_X  \:\:  x\in L_x$ (i. e., $\bigcup \mc{L}_X = X$).
\end{enumerate}
The family $\mc{L}_X$ will be called a \textit{smopology} on $X$ and
elements of $\mc{L}_X$ will be called \textit{small open} subsets  (or  \textit{smops}) in $X$.
\end{defi}

\begin{rem}
Each topological space $(X,\tau_X)$ may be expanded in many ways to  a locally small space by choosing a suitable  basis $\mc{L}_X$  of the topology $\tau_X$ such that
$\mc{L}_X$  is a sublattice in $\tau_X$ containg the empty set.
\end{rem}

\begin{defi}[{{\cite[Def. 2.9]{P}}}]
For a locally small space $(X,\mc{L}_X)$ we define the family 
$$  \mc{L}_X^{wo} =\mbox{the unions of subfamilies of }\mc{L}_X $$
of \textit{weakly open} sets.
Then $\mc{L}_X^{wo}$ is the smallest topology containing $\mc{L}_X$.
\end{defi}

\begin{exam} The nine families of subsets of $\mb{R}$  from Example   2.14  in \cite{P}
(compare Definition  1.2 in \cite{PW})  are smopologies and share the same 
family of weakly open sets (the natural topology on $\mb{R}$).
 Analogically, Definition 4.3 in \cite{PW2} shows many generalized topological spaces induced by  smopologies on $\mb{R}$ with the same family of weakly open sets (the Sorgenfrey topology). 
\end{exam}

\begin{defi}[{{\cite[Def. 31]{P-stone}}}]
A locally small space  $(X,\mc{L}_X)$ is called $T_0$ (or \textit{Kolmogorov}) if the family $\mc{L}_X$ \textit{separates points}, which means
$$\forall x,y\in X \quad (x\neq y)\implies \exists V\in  \mc{L}_X  \mid V \cap \{x,y\} \mid =1.$$
\end{defi}

\begin{defi}
Assume $(X, \mc{L}_X)$ and $(Y,\mc{L}_Y)$  are locally small spaces.
Then a mapping $f:X \to Y$ is: 
\begin{enumerate}
\item[$(a)$]  \textit{bounded} (\cite[Definition 2.40]{P}) if $\mc{L}_X$ is a refinement of $f^{-1}(\mc{L}_Y)$,
\item[$(b)$]  \textit{continuous} (\cite[Definition 2.40]{P})  if 
$f^{-1}(\mc{L}_Y) \cap_1 \mc{L}_X \subseteq \mc{L}_X $, 
\item[$(c)$] \textit{weakly continuous} if 
$f^{-1}(\mc{L}_Y^{wo})  \subseteq \mc{L}_X^{wo}$. 
\end{enumerate} 
The category of  locally small spaces and their bounded continuous mappings  is denoted by \textbf{LSS} (\cite[Remark 2.46]{P}). 
The full subcategory of $T_0$ locally small spaces is denoted by $\mathbf{LSS}_0$
(\cite[Def. 33]{P-stone}).
\end{defi}

\begin{defi} We have the following full subcategories of $\mathbf{LSS}_0$:
\begin{enumerate}
\item  the category $\mathbf{SobLSS}$ generated by the topologically sober
 objects $(X,\mc{L}_X)$, i.e., such object that the topological space $(X, \mc{L}_X^{wo})$ is sober (compare \cite[Def. 3.1]{PW}).
 \item the category $\mathbf{LKSobLSS}$ generated by the topologically locally compact sober objects. 
\end{enumerate}
\end{defi}

\begin{exam}
For $X=\mathbb{R}^2$,  we take
$$ \mc{L}_{lin}=\mbox{ the smallest smopology containing all straight lines,} $$
$$ \mc{L}_{alg}= \mbox{ the smallest smopology containing all proper algebraic subsets.}$$
Then $( \mathbb{R}^2,\mc{L}_{lin}  )$ and $( \mathbb{R}^2  ,\mc{L}_{alg}  )$
are two different topologically locally compact sober
locally small spaces with the same topology of weakly open sets (the discrete topology).
\end{exam}

Now we introduce some categories constructed from \textbf{Frm}.

\begin{defi}
The category $\mathbf{FrmS}$ consists of pairs $(L,L_s)$ with $L$ a frame and $L_s$
  a sublattice with zero \textit{sup-generating} $L$ (this means: every member of $L$ is the supremum of a subset of $L_s$) as objects  and  dominating compatible frame homomorphisms $h:(L,L_s)\to (M,M_s)$ as morphisms. Here a frame homomorphism $h:L\to M$ is called 
 \begin{enumerate}
 \item 
  \textit{dominating} if
   $\quad \forall m\in M_s \quad \exists l\in L_s\quad h(l) \wedge m = m$\\
 (then we shall  also say that $h(L_s)$  \textit{dominates} $M_s$). 
\item  \textit{compatible} if
 $\quad \forall m\in M_s \quad \forall l\in L_s \quad h(l)\wedge m \in M_s$\\
 (then we shall  also say that $h(L_s)$ \textit{is compatible with} $M_s$).
\end{enumerate}
\end{defi}

\begin{rem} \label{on}
If $h:L \to M$ is a frame homomorphism satisfying $h(L_s)=M_s$, then 
$h: (L,L_s)\to (M,M_s)$ is dominating compatible. 
\end{rem}

\begin{defi} 
The category $\mathbf{LocS}$ consists of pairs $(L,L_s)$ with $L$ a frame and $L_s$ 
 a sublattice with zero {sup-generating} $L$  as objects and
\textit{special localic maps} (i.e., right Galois adjoints $h_*:(M,M_s)\to (L,L_s)$ of dominating compatible frame homomorphisms $h:(L,L_s)\to (M,M_s)$) as morphisms.
\end{defi}

\begin{rem} \label{iso}
The categories $\mathbf{FrmS}^{op}$ and $\mathbf{LocS}$ are isomorphic.
\end{rem}

\begin{defi} We introduce the following categories:
\begin{enumerate}
\item
the  full subcategory \textbf{SpFrmS} in $\mathbf{FrmS}$
generated by objects $(L,L_s)$ where $L$ is a spatial frame, 
\item
the   full subcategory \textbf{SpLocS} in $\mathbf{LocS}$
generated  by objects $(L,L_s)$ where $L$ is a spatial frame,
\item
the full subcategory $\mathbf{ContFrmS}$ in $\mathbf{FrmS}$
generated by objects $(L,L_s)$ where $L$ is a continuous frame, 

\item the   full subcategory $\mathbf{ContLocS}$ in $\mathbf{LocS}$
generated  by objects $(L,L_s)$ where $L$ is a continuous frame.
\end{enumerate}
\end{defi}

\section{The spectral adjunction.}
\begin{thm}[the spectral adjunction] \label{adjunction}
The categories  $\mathbf{LSS}$  and $\mathbf{LocS}$ are adjoint. 
\end{thm}
\begin{proof}
\textbf{Step 1:} Defining functor $\Omega: \mathbf{LSS} \to \mathbf{LocS}$.

Functor $\Omega: \mathbf{LSS} \to \mathbf{LocS}$ is defined by
$$ \Omega (X,\mc{L}_X)=(\mc{L}^{wo}_X, \mc{L}_X), \quad 
 \Omega(f) =(\mc{L}^{wo} f)_{*} , $$
where $\mc{L}^{wo} f:\mc{L}^{wo}_Y \to \mc{L}^{wo}_X$ is given by $(\mc{L}^{wo} f)(W)=f^{-1}(W)$ for  a   bounded continuous $f:(X, \mc{L}_X)\to (Y, \mc{L}_Y)$ .

For any  locally small space $(X,\mc{L}_X)$, the pair 
$(\mc{L}^{wo}_X, \mc{L}_X)$ consists of a frame and a  sublattice  with zero that 
sup-generates the frame.

For a bounded continuous map $f:(X,\mc{L}_X) \to (Y,\mc{L}_Y)$
the  frame homomorphism $\mc{L}^{wo} f:\mc{L}^{wo}_Y \to \mc{L}^{wo}_X$ is always:
\begin{enumerate}
\item  
dominating (because $f$ is bounded): 
$$\forall W\in \mc{L}_X \ \exists V\in \mc{L}_Y\ (\mc{L}^{wo} f)(V) \cap W = W,$$
\item
compatible ({because $f$ is continuous}):
 $$\forall W\in \mc{L}_X\ \forall V\in \mc{L}_Y\ (\mc{L}^{wo} f)(V)\cap W \in \mc{L}_X.$$ 
\end{enumerate}
The mapping $(\mc{L}^{wo} f)_{*}:\mc{L}_X^{wo}\to \mc{L}_Y^{wo}$,  defined by the condition
$$(\mc{L}^{wo}f)_{*}(W)=\bigcup \{V\in \mc{L}^{wo}_Y : f^{-1}(V)\subseteq W\}=\bigcup (\mc{L}^{wo}f)^{-1}(\downarrow W),$$
is a special  localic map.
Clearly, $\Omega$ preserves identities and compositions.

\noindent  \textbf{Step 2:} Defining functor 
$\Sigma:\mathbf{LocS} \to \mathbf{LSS}$.

Functor $\Sigma:\mathbf{LocS} \to \mathbf{LSS}$ is defined by
$$\Sigma  (L,L_s)=(Spec(L),  \Delta_L (L_s) ),$$
$$ \Sigma (h_{*}) =h_{*}|_{Spec(M)}:Spec(M)\to Spec(L).$$
(Notice that $\Sigma (h_{*})$ is  always well defined by \cite[Prop.4.5]{G2}).

The pair $(Spec(L), \{ \Delta_L (a) \}_{a\in L_s})$ is always a topologically sober locally small space by Facts \ref{topol} and \ref{sober}.

For $ h:(L,L_s)\to (M,M_s)$  a dominating compatible frame homomorphism,
$\Sigma(h_*)=h_{*}|_{Spec(M)}$ is always a bounded continuous mapping between locally small spaces from $(Spec(M), \Delta_M(M_s))$ to $(Spec(L),\Delta_L(L_s))$:
\begin{enumerate}
\item  Take any $a\in M_s$. Since $h$ is dominating, for some $b\in L_s$ we have
$h_*|_{Spec(L)} (\Delta_M(a)) \subseteq h_*|_{Spec(L)}(\Delta_M(h(b))) \subseteq \Delta_L(b)$. This is why
$ h_*|_{Spec(L)}(\Delta_M(M_s))$  refines $\Delta_L(L_s)$.

\item For any $d\in L_s$, $f\in M_s$, we have
$$(h_*|_{Spec(L)})^{-1}(\Delta_L(d)) \cap \Delta_M(f) = \Delta_M (h(d) \wedge f) \in \Delta_M(M_s).$$
This is why $ (h_*|_{Spec(L)})^{-1}(\Delta_L(L_s))\cap_1  \Delta_M(M_s)\subseteq \Delta_M(M_s).$
\end{enumerate}
Clearly, $\Sigma$ preserves identities and compositions.

\noindent \textbf{Step 3:} There exists a natural transformation $\sigma$ from $\Omega \Sigma$ to $Id_{\mathbf{LocS}}$.

We define the mapping $\sigma_L: (\Delta_L(L),\Delta_L(L_s))\to (L,L_s)$  by the formula
$$\sigma_L=(\Delta_L)_*: \Delta_L(L)=\tau( \Delta_L (L_s) ) \ni A \to \bigvee \Delta_L^{-1} (\downarrow\!\! A) \in L.$$
This $\sigma_L$ is an (injective)  morphism in $\mathbf{LocS}$ since $\Delta_L$ is a (surjective) dominating compatible frame homomorphism:
\begin{enumerate}
\item  $\Delta_L (L_s)$ obviously dominates $\Delta_L(L_s)$. 
\item  Take any $D\in \Delta_L (L_s)$ and $f\in L_s$. Choose $d\in L_s$ such that $\Delta_L(d)=D$. Then $D\cap \Delta_L(f)=\Delta_L (d\wedge f) \in \Delta_L (L_s)$, so $\Delta_L$ is compatible.
\end{enumerate}
For a special localic map $h_*:(M,M_s)\to (L,L_s)$ and $a\in M$,
we have 
$$\sigma_L \circ \Omega\Sigma(h_*) (\Delta_M(a))=\sigma_L \circ (\mc{L}^{wo}h_*|_{Spec(M)})_{*} (\Delta_M(a))=$$
$$\sigma_L(\bigvee \{Z\in \Delta_L (L): (h_*|_{Spec(M)})^{-1}(Z)\subseteq \Delta_M(a)\})=
h_*(a)= h_*\circ \sigma_M  (\Delta_M(a)),$$
so $\sigma$ is a natural transformation.

\textbf{Step 4:} There exists a natural transformation $\lambda $ from $Id_{\mathbf{LSS}}$ to $ \Sigma\Omega$.

W define $\lambda_X: (X,\mc{L}_X)\to (Spec(\mc{L}^{wo}_X),\Delta(\mc{L}_X))$,
where $\Delta=\Delta_{\mc{L}^{wo}_X}$, by 
$$\lambda_X:X\ni x \mapsto \ext_X \{x\} \in Spec(\mc{L}^{wo}_X),$$
which is a bounded continuous map:
\begin{enumerate}
\item  Take any $W\in \mc{L}_X$. Then $\lambda_X(W)=\{ \ext\{x\}:x\in W\}$, which is contained in $\Delta(W)=\{ V\in Spec(\mc{L}^{wo}_X): W \not\subseteq V \}$.  
This is why
$ \lambda_X(\mc{L}_X) \mbox{ refines } \Delta(\mc{L}_X)$. 

\item  Take any $W\in \mc{L}_X$. Since $x\in W$ iff $\ext\{x\}\in \Delta(W)$, 
we have $W=\lambda_X^{-1}(\Delta(W))$. This is why
$\lambda_X^{-1} (\Delta(\mc{L}_X))\cap_1  \mc{L}_X\subseteq \mc{L}_X$.
\end{enumerate}
For a bounded continuous map $f:(X,\mc{L}_X) \to (Y,\mc{L}_Y)$, we have 
$$(\Sigma\Omega(f)\circ \lambda_X)(x)=(\mc{L}^{wo}f)_{*} (\ext_X\{ x\} )=
\bigcup \{ Z\in Spec(\mc{L}^{wo}_X): x\notin f^{-1}(Z)\}$$
$$=\ext_Y\{f(x)\}=(\lambda_Y \circ f)(x).$$
 This means $\lambda$ is a natural transformation. 

\textbf{Step 5:} Functor $\Omega$ is a left adjoint of functor $\Sigma$.

We are to prove that
$$\sigma_L|_{Spec(\Delta(L))}\circ \lambda_{Spec(L)}=id_{Spec(L)}, \qquad
\sigma_{\mc{L}^{wo}_X}\circ (\mc{L}^{wo}\lambda_X)_*=id_{\mc{L}^{wo}_X}.$$
For $p\in Spec(L)$, we have
$$  \sigma_L|_{Spec(\Delta(L))}\circ \lambda_{Spec(L)}(p)=\sigma_L|_{Spec(\Delta(L))}(\ext_{Spec(L)}\{ p\})=id_{Spec(L)}(p).$$

For $W\in \mc{L}^{wo}_X$, we have
$$\sigma_{\mc{L}^{wo}_X}\circ (\mc{L}^{wo}\lambda_X)_*(W)=\sigma_{\mc{L}^{wo}_X}
(\Delta_{\mc{L}^{wo}_X}(W))=id_{\mc{L}^{wo}_X}(W).$$
\end{proof}



\section{A Stone-type duality.}

\begin{thm} \label{stone}
The categories $\mathbf{SobLSS}$, $\mathbf{SpLocS}$ and $\mathbf{SpFrmS}^{op}$ are  equivalent.
\end{thm}
\begin{proof} Assume $(X,\mc{L}_X)$ is an object of $\mathbf{SobLSS}$.
Then $\lambda_X:(X,\mc{L}_X) \to (Spec(\mc{L}_X^{wo}),\Delta(\mc{L}_X))$  is a homeomorphism
by \cite[Ch. II, Prop. 6.2]{PP} and $\lambda_X(\mc{L}_X)=\Delta(\mc{L}_X)$. Hence
 $\lambda_X$ is an isomorphism in $\mathbf{SobLSS}$.

Assume $(L,L_s)$ is an object of $\mathbf{SpFrmS}$. Then,  by Fact \ref{iso-spatial},  
$\Delta_L$ is  an isomorphism of frames
and  $\Delta_L^{-1}(\Delta_L(L_s) )=L_s$, so, by Remark \ref{on}, both $\Delta_L: (L,L_s)\to (\Delta_L(L), \Delta_L(L_s) )$ and $\Delta_L^{-1}$ are dominating compatible frame homomorphisms.
Hence  $\sigma_L=(\Delta_L)_*$ is an isomorphism in $\mathbf{SpLocS}$. 

 Restricted $\sigma:\Omega\Sigma  \rightsquigarrow  Id_{\mathbf{SpLocS}}$ 
 and $\lambda: Id_{\mathbf{SobLSS}} \rightsquigarrow \Sigma\Omega$  are  natural isomorphisms. Hence
$\mathbf{SobLSS}$ and $\mathbf{SpLocS}$ are equivalent.
Similarly to Remark \ref{iso}, categories $\mathbf{SpLocS}$ and $\mathbf{SpFrmS}^{op}$ are isomorphic.
\end{proof}

\section{Hofmann-Lawson duality.}
In this section we give a new version of Theorem \ref{HL}.
\begin{lem} \label{5.5}
Let $(L,L_s)$ be an object of $\mathbf{ContFrmS}$. Then
 $(Spec(L),\Delta(L_s))$ is an object of $\mathbf{LKSobLSS}$.
\end{lem}
\begin{proof}  
By the proof of Theorem \ref{adjunction},  $(Spec(L),\Delta(L_s))$ is a locally small space.
By Facts \ref{sober} and \ref{6.4.2}, this space is topologically sober locally compact.
\end{proof}

\begin{lem} \label{5.6}
For an object $(X,\mc{L}_X)$ of $\mathbf{LKSobLSS}$,
 the pair $(\mc{L}_X^{wo},\mc{L}_X)$ is an object of $\mathbf{ContFrmS}$ (so also of $\mathbf{ContLocS}$).
\end{lem}
\begin{proof} 
By the proof of Theorem \ref{adjunction}, $(\mc{L}_X^{wo},\mc{L}_X)$ 
is an object of $\mathbf{SpFrmS}$. By Fact \ref{5.1.1}, $\mc{L}_X^{wo}$ is a continuous frame.  
\end{proof}

\begin{thm}[Hofmann-Lawson duality for locally small spaces]
The categories
$\mathbf{LKSobLSS}$, 
$\mathbf{ContLocS}$ and $\mathbf{ContFrmS}^{op}$
are equivalent.
\end{thm}
\begin{proof} By lemmas \ref{5.5}, \ref{5.6}, the restricted functors 
$$\Sigma: \mathbf{ContLocS} \to \mathbf{LKSobLSS}, \qquad 
\Omega:\mathbf{LKSobLSS} \to \mathbf{ContLocS}$$ 
are well defined. 
 The restrictions  
 $$\lambda: Id_{\mathbf{LKSobLSS}} \rightsquigarrow \Sigma\Omega ,\qquad 
 \sigma:  \Omega\Sigma  \rightsquigarrow  Id_{\mathbf{ContLocS}}$$ 
 of natural  isomorphisms from Theorem \ref{stone} are natural isomorphisms.
 Obviously, $\mathbf{ContLocS}$ and $\mathbf{ContFrmS}^{op}$ are isomorphic.
\end{proof}

\begin{rem} Further equivalences of categories may be obtained using Exercise V-5.24 in \cite{G2} (or Exercise V-5.27 in \cite{CLD}).
\end{rem}

\begin{exam}
Consider the space $(\mb{R}, \mc{L}_{l^+om})$ from \cite[Ex. 2.14(4)]{P} where
$$\mc{L}_{l^+om}=\mbox{the finite unions of bounded from above open intervals}$$
and the following functions:
\begin{enumerate}
\item the function $-\id: \mb{R} \ni x \mapsto -x \in \mb{R}$ is  continuous but is not bounded, so $\mc{L}^{wo}(-\id):(\tau_{nat},\mc{L}_{l^+om}) \to (\tau_{nat}, \mc{L}_{l^+om})$ is a compatible frame homomorphism but is not dominating,

\item the function $\sin: \mb{R} \ni x \mapsto \sin(x) \in \mb{R}$ is bounded an weakly continuous but not  continuous, so  $\mc{L}^{wo}\sin:(\tau_{nat},\mc{L}_{l^+om}) \to (\tau_{nat}, \mc{L}_{l^+om})$ is a dominating frame homomorphism but is not compatible,

\item the function $\arctan: \mb{R} \ni x \mapsto \arctan(x) \in \mb{R}$ is bounded continuous, so  $\mc{L}^{wo}\arctan:(\tau_{nat},\mc{L}_{l^+om}) \to (\tau_{nat}, \mc{L}_{l^+om})$ is a dominating compatible frame homomorphism,

\item the function $\frac{1}{\exp}: \mb{R} \ni x \mapsto \exp(-x) \in \mb{R}$
is continuous but not bounded, so  the mapping $\mc{L}^{wo}\frac{1}{\exp}:(\tau_{nat},\mc{L}_{l^+om}) \to (\tau_{nat}, \mc{L}_{l^+om})$ is a compatible frame homomorphism but is not dominating.
\end{enumerate}
\end{exam}

\end{document}